\documentclass[a4paper,10pt]{amsart}
\usepackage[plainpages=false]{hyperref}
\usepackage{amsfonts,amssymb,amsthm, mathrsfs, lscape}
\usepackage{verbatim}

\usepackage{latexsym}
\usepackage{amssymb}
\usepackage{graphics}
\usepackage{graphicx}
\usepackage[space]{cite}

\usepackage{latexsym}
\usepackage{amsfonts}
\usepackage{amsmath}
\usepackage{latexsym}
\usepackage{amsfonts}
\usepackage{amssymb}
\usepackage{rawfonts}

\usepackage[usenames,dvipsnames]{pstricks}
\usepackage{epsfig}
\usepackage{pst-grad} 
\usepackage{pst-plot} 
\usepackage[space]{grffile} 
\usepackage{etoolbox} 
\makeatletter 

\renewcommand{\Re}{{\operatorname{Re}\,}}

\renewcommand{\epsilon}{\varepsilon}

\newcommand{\kahler}{K\"ahler }

\newcommand{\N}{{\mathbb N}}

\newcommand{\R}{{\mathbb R}}
\newcommand{\C}{{\mathbb C}}
\newcommand{\D}{{\mathbb D}}

\newcommand{\dbar}{\bar\partial}

\newcommand{\onep}{1+\epsilon(k)}
\newcommand{\ok}{O(\frac{1}{k})}

\newcommand{\epk}{\epsilon(k)}

\renewcommand{\phi}{\varphi}

\newcommand{\hcal}{\mathcal{H}}

\newcommand{\tlog}{e^{-2(\log k)^2}}
\newtheorem{thm}{Theorem}[section]
\newtheorem{theorem}{{Theorem}}[section]
\newtheorem{theo}{{Theorem}}[section]
\newtheorem{cor}[theorem]{{Corollary}}
\newtheorem{lem}[theorem]{{Lemma}}
\newtheorem{prop}[theorem]{{Proposition}}

\newenvironment{rem}{\medskip\noindent{\it Remark:\/} }{\medskip}
\newenvironment{nota}{\medskip\noindent{\it Notation:\/} }{\medskip}
\newenvironment{note}{\medskip\noindent{\it Notice:\/} }{\medskip}
\newtheorem{defin}[theo]{{\sc Definition}}

\theoremstyle{definition}

\numberwithin{equation}{section}

\def \N {\mathbb N}

\def \C {\mathbb C}

\def \R {\mathbb R}

\title[Bergman Kernel]{Estimations of the Bergman Kernel of the punctured disk}

\author{Jingzhou Sun}
\thanks{The author is partially supported by NNSF of China no.11701353 and the STU Scientific Research Foundation for Talents no.130/760181.}
\address{Department of Mathematics, Shantou University, Shantou City, Guangdong Province 515063, China}
\email{jzsun@stu.edu.cn}

\begin{document}
	
	\begin{abstract}
Using the techniques developed in \cite{SunSun}, we give estimations of the Bergman kernel of the punctured disk with the standard complete Poincar\'{e} metric. As an application, we improve the result of \cite{AMM} on the Bergman kernels of punctured Riemann surfaces near singularities.
	\end{abstract}

	\maketitle
	
	

	\section{Introduction}

	\
	The asymptotic of the Bergman kernel of an ample line bundle over a projective manifold, since being proved by Tian, Zelditch, Catlin, Lu\cite{Tian1990On, Zelditch2000Szego, Lu2000On, Catlin1999, MM}, has played an important role in the recent developments in complex geometry, for example \cite{donaldson2001}\cite{Sun2011Expected}  \cite{Donaldson2014Gromov},	\cite{Donaldson15}. 
	
	It is then natural to consider the Bergman kernels of singular metrics. In \cite{AMM},	Auvray-Ma-Marinescu studied the Bergman kernels on punctured Riemann surfaces and showed that near each singularity the Bergman kernel is very close to the Bergman kernel of the standard punctured disk. They also showed some properties of the Bergman kernel of the standard punctured disk. There are also many results in the literature studying the asymptotics of Bergman kernels of singular K\"ahler metrics, see for example  \cite{LL, RT, DLM}. 
	
	In \cite{SunSun}, we developed the mass concentration technique. But we were concerned only with some integrals over the Fubini-Study metric, and in the end we did not need to figure out the asymptotic of the Bergman kernel. The author realizes that our technique can be used to give a clearer description of the asymptotic of the Bergman kernel of the punctured disk and of the punctured Riemann surfaces. 
	
	In this article, we will first focus on the punctured disk, where we can give a very clear description of the Bergman kernel. Then we put ourself under the same scenario as that considered in \cite{AMM}, and show that we have a better understanding of the Bergman kernels of punctured Riemann surfaces. 
	
	As noticed in \cite{SunSun}, the Begman kernel of the punctured disk should behave according to the distance to the origin. To distinguish it from that of the punctured Riemann surfaces, we denote by $\rho_{0,k}$ the Bergman kernel of the punctured disk with the setting explained in section 2.
	
	For $z$ near the origin, we have the following results.
	\begin{thm}\label{lattice-inside}
		Let $k\geq 3$ be an integer.
		\begin{itemize}
			\item[$\bullet$]  When $|z|\geq e^{-k/2}$
			$$	\rho_{0,k+1}(z)=\frac{(\log \frac{1}{|z|^2})^{k+1}|z|^2}{2\pi(k-1)!}(1+\epsilon_1)$$
			, where $0<\epsilon_1\leq 2^{k+1}|z|^2$.
			\item[$\bullet$] Let $b\geq 2$ be an integer such that $\frac{k}{b(b+1)}\geq \log 2$. When  $|z|=e^{-k/(2b)}$, we have
			$$	\rho_{0,k+1}(z)=\frac{k^{k+1}e^{-k}}{2b\pi(k-1)!}(1+\epsilon_b)$$
			, where $(1+\frac{1}{b})^ke^{-k/b}+(1-\frac{1}{b})^ke^{k/b}<\epsilon_b<2[(1+\frac{1}{b})^ke^{-k/b}+(1-\frac{1}{b})^ke^{k/b}]$
		\end{itemize}	
	\end{thm}	
	Since $(1+\frac{1}{b})^ke^{-k/b}< e^{-k/(2b^2)+k/(3b^3)}$, the error term is very small when $b$ is small, that is when $z$ is close to the origin. For example, when $b<\sqrt{k}/(\log k)$, the error is less than $e^{-k/3}$. 
	One can also use Stirling's formula on factorial to make the formula looks even clearer, at the cost of some precision. Recall that Robbins \cite{robbins} proved that 
	$$e^{1/(12n+1)}<\frac{n!}{\sqrt{2\pi}n^{n+1/2}e^{-n}}<e^{1/(12n)}$$
	So we have	
	\begin{cor}
		Assume $k\geq 79$, $b$ a positive integer satisfying $\frac{k}{2b^2}-\frac{k}{3b^3}\geq 3\log k$. When  $|z|=e^{-k/(2b)}$, we have
		$$	\rho_{0,k+1}(z)=\frac{k^{3/2}}{b(2\pi)^{3/2}}(1+\frac{1}{12k}+\epsilon_b)$$
		, where $\epsilon_b<\frac{9}{k^2}$
	\end{cor}
Points where $|z|=e^{-k/(2b)}$, with $b$ an integer, are called lattice points.
And for $z$ between lattice points, we have
\begin{thm}\label{others-inside}
	Let $a\geq 1$ be an integer such that $a\leq \sqrt{k}/\log k-1$, $t=-\log |z|^2$ and $k\geq 55$. Then for $t\in (\frac{k}{a+1},\frac{k}{a})$, $\rho_{0,k+1}$, up to an error smaller than
	$$2[(1+\frac{1}{a})^ke^{-k/a}]+2[(1+\frac{1}{a+1})^ke^{-k/(a+1)}]$$
	, is convex, and has exactly one minimum, which is smaller than 
	$$e^{-k/(17a^2)}\frac{k^{k+1}e^{-k}}{2\pi (k-1)!}(\frac{1}{a}+\frac{1}{a+1})$$
\end{thm}

	We see that, due to the sparseness of sections near the origin, the Bergman kernel exhibits a very interesting quantum phenomenon, namely, it attains the "regular" value $\frac{k^{1/2}\log\frac{1}{|z|^2} }{(2\pi)^{3/2}}$ at the "lattice points", and it decreases to nearly $0$, for $k$ large, between the "lattice points".
	
\

 For $z$ far away from the origin, which we refer to as "outside", we have the following results.

 \begin{thm}\label{outside}
	With the notation $-\log |z|^2=t$, we have
	$$\rho_{0,k+1}(z)=\frac{k}{2\pi}(1+\epsilon_t)$$
	where $\epsilon_t<(\frac{t^2}{\pi^2 (k-1)}+2\frac{k+1}{k-1})(1+(\frac{2\pi }{t} )^2)^{-(k+1)/2} $
\end{thm}
Clearly, the error term decreases as $t$ decreases. In particular, when $t$ is small enough, this error becomes surprisingly small. For example, when $\frac{2\pi }{t}\geq e$, $\epsilon_t<e^{-k}$. One can compare this result with the asymptotic expansion in the compact smooth case, where by \cite{Lu2000On} the coefficients are functions of derivatives of the Riemann curvature. Under our setting, the curvature is constant, so by Lu's result, the coefficient of $k^{-1}$ is $-1$ and the coefficient of $k^{-i}$ for $i>1$ should be $0$. So our result agrees with this as $e^{-k}$ is of course asymptotically smaller than any $k^{-i}$. But the asymptotic expansion neither implies exponential decay nor tells when $k$ is big enough, while our theorem does both. The following corollary is interesting in that it does not require $k$ to be large.
\begin{cor}For $k\geq 3$
	$$	\lim_{|z|\to 1}\rho_{0,k}(z)=\frac{k}{2\pi }$$
\end{cor}
When $b\geq \frac{\sqrt{k\log k}}{2\pi }$, we see that the error is less than $k^{-1/2}$. 
This means that when $\log \frac{1}{|z|^2}\leq \frac{2\pi k^{1/2}}{\sqrt{\log k}}$, the Bergman kernel is basically $\frac{k}{2\pi}$.

\

We are left with a gap, where none of the above results produce satisfactory estimates. It was called the "neck" in \cite{SunSun}. We show that the Bergman kernel over the "neck" area is still understandable with the following theorem
	\begin{thm}\label{middle}
		Let $b$ be an integer such that $3<b\leq \sqrt{k}\log k$. Let $\gamma_b(u)=\sum_{c=-\infty}^{\infty}e^{-\frac{k}{2}(\frac{c}{b}-u)^2}$
		\begin{itemize}
			\item[$\bullet$]  If $|z|=e^{-k/(2b)}$, then
			$$\rho_{0,k+1}(z)\geq \frac{e^{-1/(12k)}}{b}(\frac{k}{2\pi})^{3/2}[(1-\frac{8(\log k)^4}{k})\gamma_b(0)-(6+2e^{\frac{8(\log k)^3}{3\sqrt{k}}})k^{-2\log k}]$$
			and 
			$$\rho_{0,k+1}(z)\leq \frac{1}{b}(\frac{k}{2\pi})^{3/2}\gamma_b(0)$$
			\item[$\bullet$] If $t=\log \frac{1}{|z|^2}\in (\frac{k}{b+1},\frac{k}{b})$, then
			$$\rho_{0,k+1}(z)\geq  \frac{tke^{-1/(12k)}}{(2\pi)^{3/2}}[(1-\frac{(\log k)^3}{3\sqrt{k}})\gamma_b(u)-12e^{-\frac{k}{2}(\frac{\log k}{\sqrt{k}}-u)^2}]$$
			$$\rho_{0,k+1}(z)\leq  \frac{tk}{(2\pi)^{3/2}}[(1+\frac{(\log k)^3}{3\sqrt{k}})\gamma_b(u)]$$
			,	where $u=1-\frac{tb}{k}$.
		\end{itemize}
		
	\end{thm}
	\begin{figure}
	\begin{center}
		\includegraphics[width=0.8 \columnwidth]{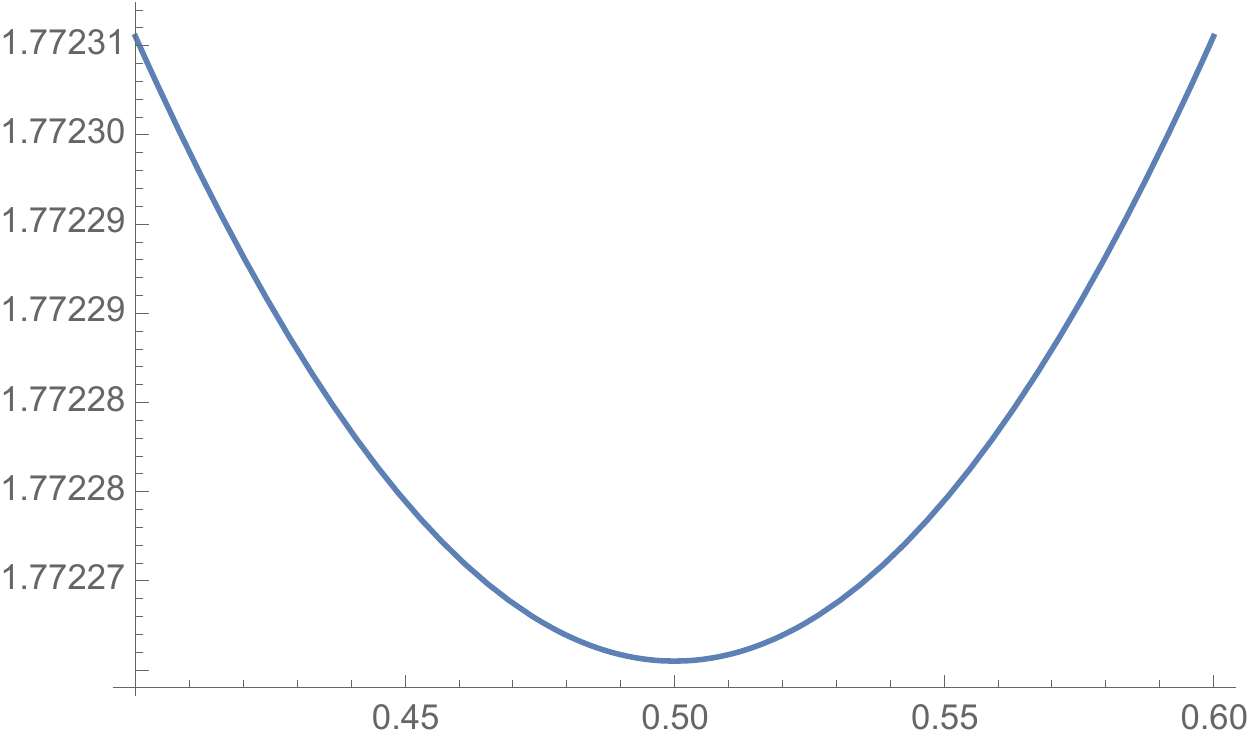}
		\caption{middle of the neck}\label{graph-turb}
	\end{center}
	
\end{figure}
	This theorem says that in the "neck" area, the Bergman kernel $\rho_{0,k+1}$ behaves very much like the function $\gamma_b$. We show the reader the graph (figure \ref{graph-turb}) of 
		the function $h(x)=\sum_{c=-\infty}^{\infty}e^{-(c-x)^2/2}$ between the lattice points, where one can tell that the amplitude is strictly $O(1)$, independent of $k$. So near the points (middle of the "neck") where $\log |z|^2=-\sqrt{k}$, the Bergman kernel behaves almost independent of $k$. We expect the Bergman kernel to display milder and milder turbulence during the transition from the "inside" to the "outside". 
		
		\

	The basic setting in \cite{AMM} is as follows:

	Let $X$ be a compact Riemann surface and let $D=\{a_1,\cdots,a_N\}\subset X$ be a finite set. We consider the punctured Riemann surface $X_D=X\backslash D$ and a Hermitian form $\omega$ on $X_D$. Let $L$ be a holomorphic line bundle of degree $d>0$ on $X$, and let $h$ be a singular Hermitian metric on $L$ such that:
	\begin{itemize}
		\item[$\alpha$)] $h$ is smooth over $X_D$, and for all $j=1,\cdots, N$, there is a trivialization of $L$ in the complex neighborhood $V_j$ of $a_j$ in $X$, with associated coordinate $z_j$ such that $|1|_h^2(z_j)=|\log(|z_j|^2)|$.
		\item[$\beta$)] There exists $\epsilon>0$ such that the (smooth) curvature $R^L$ of $h$ satisfies $iR^L\geq \epsilon \omega$ over $X_D$ and $iR^L=\omega$ on $V_j\backslash\{a_j\}$; in particular, $\omega=\omega_0$ in the local coordinate $z_j$ on $V_j$, where $\omega_0$ is the standard Poincar\'{e} metric on the punctured unit disc.
	\end{itemize}
	For each $j$, there exists a radius $R_j>0$ such that the disc $|z_j|\leq R_j$ is contained in $V_j$. For simplicity of arguments, we will
	assume that $R_j\leq 1/2$ for each $j$.

	By assumption $\beta$, we see that there exists a real number $\lambda$ such that 
	$$Ric(\omega)\geq \lambda \omega$$
	We will consider the Bergman space $\hcal_k$ consisting of holomorphic sections of $L^k$ on $X_D$ such that
	$$\int_{X_D}|s|_h^2\omega<\infty$$
	And the Bergman kernel $\rho_k$ of $\hcal_k$ is defined as
	$$\rho_k=\sum |s_i|^2_h$$
	, where $(s_i)$ is an orthonormal basis of $\hcal_k$. To distinguish notations, we will use $\rho_{0,k}$ to denote the Bergman kernel for the standard punctured disk.

\begin{thm}\cite{AMM}
	Assume that $(X,\omega,L,h)$ fulfill conditions $(\alpha)$ and $(\beta)$. Then the following estimate holds: for every integer $l,m\geq 0$, and every $\delta>0$, there exists a constant $C=C(l,m,\delta)$ such that for all $k\in \N^*$, and $z\in V_1\Cup\cdots\Cup V_N$ with the local coordinate $z_j$,
	$$|\rho_k-\rho_{0,k}|_{C^m}(z_j)\leq Ck^{-l}|\log(|z_j|^2)|^{-\delta}$$
	, where the notation $|\cdot|_{C^m}$ means
	$$|f|_{C^m}(x)=(|f|+|\nabla f|_\omega+\cdots+|(\nabla)^m f|_\omega)(x)$$ 
	, for any local function $f$ on $X$. 
\end{thm}

	\begin{defin}
		We define $W_k\subset \D$ to be the set consisting of the point $z$ such that
		$$|z|\leq e^{-(k-2)^{3/8}}$$
	\end{defin}

	We will use different assumptions on $k$ in order to have simpler formulas. For clarity, we collect them here.
	\begin{itemize}
		\item[1. ] $\log k+\log(-\log R)-\log (\epsilon+\lambda/k)>0$
		\item[2. ] $ k\geq (-9\log (R/2))^4$
		\item[3. ] $k\geq 23190$.
	\end{itemize}
Then we can state our theorem on punctured Riemann surfaces.
\begin{thm}\label{thm-riemannsurface}
	Around each singularity $a_i$, let $z=z_i$, $R=R_i$. Then for k satisfying assumptions 1 to 3, and for $z\in W_k$, we have
	$$	|\rho_k-\rho_{0,k}|	\leq 50(1+d)\frac{k^{-(k-5)/8}}{\sqrt{\epsilon+\lambda/k}}(-2e\log (R/2))^{k/2}(\frac{R}{2})^{k^{3/4}}\rho_{0,k}$$
\end{thm}
Without the consideration of the derivatives, our result is an obvious improvement of that of \cite{AMM}. One can see that our constant decays much faster than $e^{-k}$. Also, in order for our estimates to be valid, we do not need our $k$ to be too large. In fact, given the numbers $\epsilon, 
\lambda, R$, one can quickly decide the bound of $k$ for the estimate to hold. This kind of effective bounds is very rare in this field. Also, since we have very detailed understanding of $\rho_{0,k}$ in sections 2 and 3, this theorem says that both the absolute and relative errors are very small. Indeed, theorem \ref{lattice-inside} says that when $-\log |z|^2\geq k-1$, $\rho_{0,k}$ behaves almost the same as $\frac{(-\log |z|^2)^k|z|^2}{2\pi(k-2)!}$, which of course decays faster than $(-\log|z|^2)^{-\delta}$ for any $\delta<\infty$. Also the seemingly restrictive condition that $z\in W_k$ actually is good enough. The reason is that outside $W_k$, the Bergman kernel already have good asymptotic as that in the case of smooth metric over a compact manifold, which can be proved using Lu's method of peak section, using the fact that outside $W_k$ the injective radius of $X_D$ is big enough as explained in \cite{SunSun}. We should assure the reader that $W_k$ can be made bigger if one wish, by modifying the indices in our proof.

We then continue to show the following theorem for the first derivative of the difference $\rho_k-\rho_{0,k}$. 

\begin{thm}\label{nabla}
	With the notations and assumptions in theorem \ref{thm-riemannsurface} and  $2k^{3/4}\geq \frac{\log(\epsilon+\lambda/k)}{\log (R/2)}$, we have
	$$	|\nabla (\rho_k-\rho_{0,k})|\leq (1+\frac{d}{k^2})25\sqrt{2}t\frac{k^{-(k-10)/8}}{\sqrt{\epsilon+\lambda/k}}(-2e\log (R/2))^{k/2}(\frac{R}{2})^{k^{3/4}}\rho_{0,k} $$
	, where $t=-\log |z|^2$
\end{thm}
This result is still stronger than that in \cite{AMM} as explained for theorem \ref{thm-riemannsurface}. Our purpose to prove this theorem is to show the reader that, if desired, one can get estimates with high precisions for higher order derivatives of $\rho_k-\rho_{0,k}$, and to do so, one just need to follow the same route as we prove this theorem, and one will only have to deal with combinatorial problems in order to give general formulas.

Besides these main results, we would like also to mention lemma \ref{linfinity}, which uses $L^2$ norm to bound point-wise norm for holomorphic sections, near the singularity. We hope that the lemma itself can be found useful by the readers.


	The structure of this article is as follows. The first section is devoted to the case of the standard the punctured disk, where we show that we can describe the Bergman kernel very well. The last section is devoted to the general case of punctured Riemann surfaces, where we prove theorem \ref{thm-riemannsurface} and theorem \ref{nabla}.

	\textbf{Acknowledgements.} The author would like to thank Professor Song Sun for so many nonstop insightful discussions. The author would also like to thank Professor Bernard Shiffman for his continuous and unconditional support.
	
	\section{punctured disc}

The standard complete Poincar\'{e} metric on the punctured disk $\mathbb D^*=\{z\in \C||z|\leq 1\}$ is given by the \kahler form
	\begin{equation} \label{eqn2-1}
	\omega=\frac{i dz\wedge d\bar{z}}{|z|^2(\log \frac{1}{|z|^2})^2}.
	\end{equation}
	The corresponding K\"ahler potential is $\Phi=-\log \log \frac{1}{|z|^2}$, and the scalar curvature of $\omega$ is $-2$. For $k\geq 1$, we let $\hcal_{0,k}$ be the Bergman space of holomorphic functions $f$ on $\mathbb D^*$ such that 
	$$\|f\|_k^2:=\int_{\mathbb D^*} |f|^2 e^{-k\Phi}\omega<\infty. $$
	On $\hcal_{0,k}$ we denote by $\langle\cdot, \cdot\rangle_k$ the corresponding Hermitian inner product.

	\begin{lem}\label{integral} For any $a\geq 1$, we have $z^a\in \hcal_{0,k}$ and 
		\begin{equation} \label{eqn2-2}
		\langle z^a, z^b\rangle_k=\frac{2\pi(k-2)!}{ a^{k-1}}\delta_{ab}. 
		\end{equation}
		In particular, the functions $\{(\frac{ a^{k-1}}{2\pi(k-2)!})^{1/2}z^a|a\geq 1\}$  form an orthonormal basis of $\hcal_{k}$
	\end{lem}
	
the Bergman kernel of $\hcal_{0,k}$ is given by 
\begin{equation} 
\rho_{0,k}=\frac{(\log 1/|z|^2)^k}{2\pi(k-2)!}\sum_{a=1}^{\infty}a^{k-1}|z|^{2a}.
\end{equation}

We will use the notations $x=|z|^2$, $t=-\log x$, so
$$\rho_{0,k+1}=\frac{t^{k+1}}{2\pi(k-1)!}\sum_{a=1}^{\infty}a^{k}x^a$$


For each integer $b\ge 1$, we can write 
$$\sum_{a=1}^{\infty}a^{k}x^a=b^k x^b\sum_{c=1-b}^{\infty}(1+\frac{c}{b})^k e^{-ct}$$

We will denote by $$f_b(t)=\sum_{c=1-b}^{\infty}(1+\frac{c}{b})^k e^{-ct}$$
Let $d=\frac{bt}{k}$, 
Then $$t^{k+1}b^kx^b=tk^ke^{-k}(de^{1-d})^k$$
We define  $g(d)=de^{1-d}$, then 
\begin{equation}\label{rk}
\rho_{0,k+1}=\frac{tk^ke^{-k}}{2\pi (k-1)!}(g(d))^kf_b(t)
\end{equation}

We look at the exponent of a general summand of $f_b$, which is  $h_t(c)=k\log(1+\frac{c}{b})-ct$. Clearly $h_t$ is a concave function of $c$. When $t=\frac{k}{b}$,  the only maximum of $h_t$ is attained at $c=0$. The derivatives are
$h_t'(1)=\frac{-k}{b(b+1)}$ and $h_t'(-1)=\frac{k}{b(b-1)}$. So we can use power series to estimate the summation, and have the following two lemmas:
\begin{lem}
	$$0<f_1(k/b)-1<2^ke^{-k}/(1-e^{-k/2})$$
\end{lem}

\begin{lem}For $b\geq 2$, an integer, we have
	$$0<f_b(k)-1<\frac{(1+1/b)^ke^{-k/b}}{1-e^{-\frac{k}{b(b-1)}}}+\frac{(1-1/b)^ke^{k/b}}{1-e^{-\frac{k}{b(b-1)}}}$$
\end{lem}

One can see that this lemma is particularly useful when $b$ is small, because then $e^{-\frac{k}{b(b+1)}}$ is very small. At the lattice points, where $t=\frac{k}{b}$, we have $d=1$. So we have proved theorem \ref{lattice-inside}. Notice that the number 79 is to make the condition $\frac{k}{2b^2}-\frac{k}{3b^3}>3\log k$ not empty. 

Theorem \ref{lattice-inside} also tells us that between the lattice points, we can use only two terms with the consequence of a small error. Namely for $t\in [\frac{k}{a+1},\frac{k}{a}]$, we can use $t^{k+1}(a^ke^{-at}+(a+1)^ke^{-(a+1)t})$ to approximate $\rho_{0,k}$, with error less than 
$E_a+E_{a+1}$
, where $E_a=2[(1+\frac{1}{a})^ke^{-k/a}]$. 

We denote by $h_a(t)=t^{k+1}a^ke^{-at}$, then $h_a'(t)=a^kt^ke^{-at}(k+1-at)$ and 
$$h_a''(t)=a^kt^{k-1}e^{-at}[k(k+1)-(2k+2)at+(at)^2]$$
So we have:
\begin{lem}
	$h_a''(t)>0$ when $at>k+2$ or $at<k$
\end{lem}
We assume $a\leq \sqrt{k}/\log k$, then we can calculate
\begin{itemize}
\item[$\bullet$]When $t=k/a$, $(h_a+h_{a+1})'>0$.
\item[$\bullet$] When $t=k/(a+1)$,  $(h_a+h_{a+1})'<0$.
\item[$\bullet$]When $t\in (\frac{k}{a+1},\frac{k+2}{a+1})$, $(h_a+h_{a+1})''>0$. So for
$t\in (\frac{k}{a+1},\frac{k}{a})$, $(h_a+h_{a+1})''>0$. So $h_a+h_{a+1}$ has exactly one minimum in this interval.
\end{itemize}
Then we check that when $t=\frac{k+2}{a+1}$, $(h_a+h_{a+1})'<0$ when $k\geq 55$. So the minimum lies in the interval $ (\frac{k+2}{a+1},\frac{k}{a})$.

It seems complicated to find the minimum, but we can estimate the value of $h_a+h_{a+1}$ when $t=k\log \frac{a+1}{a}$. We denote by $s_a=k\log \frac{a+1}{a}$
 Then we calculate $\log [(\frac{s_a}{k/a})^{k+1}e^{-as_a+k}]$. We can expand it as a Taylor series of $y=\frac{1}{a}$ to get
$$k(-\frac{y^2}{8}+\frac{y^3}{8}-\frac{45 y^4}{576}+O(y^5))$$
So for $a\geq 2$, we see that 
$$(\frac{s_a}{k/a})^{k+1}e^{-as_a+k}\leq e^{-\frac{k}{16a^2}}$$
Similarly, we can estimate the ratio $(\frac{s_a}{k/(a+1)})^{k+1}e^{-(a+1)s_a+k}$, to see that the the ratio 
$$(\frac{s_a}{k/(a+1)})^{k+1}e^{-(a+1)s_a+k}\leq e^{-\frac{k}{16a^2}}$$
So for $a\geq 2$, the minimum of $h_a+h_{a+1}$ for $t\in (\frac{k}{a+1},\frac{k}{a})$, is less than $e^{-k/(66a^2)}(h_a(\frac{k}{a})+h_{a+1}(\frac{k}{a+1}))$. 

For $a=1$, we can easily see that the mimimums of the two ratios are both less than $e^{-0.059k}$.
So we have
\begin{prop}
	Let $a\geq 1$ be an integer such that $a\leq \sqrt{k}/\log k$, $t=-\log |z|^2$ and $k\geq 55$. Then for $t\in (\frac{k}{a+1},\frac{k}{a})$, $h_a(t)+h_{a+1}(t)$ is convex, and has exactly one minimum, which is less than 
	$$e^{-k/(17a^2)}(h_a(\frac{k}{a})+h_{a+1}(\frac{k}{a+1}))$$
\end{prop}
So we have proved theorem \ref{others-inside}.

\

As $b$ gets bigger, the error grows. When the error is unacceptable, we need a new method to estimate $\rho_{0,k+1}$. We recall the Poisson summation formula. 

Let $$ \hat{f}(\xi)=\int_{-\infty}^{\infty}e^{-2i\pi \xi x}f(x)dx $$
denote the Fourier transform of $ f(x)$, then for suitable function $f(x)$, the
Poisson's summation formula says:
$$\sum_{c=-\infty}^{\infty}f(c)=\sum_{\xi=-\infty}^{\infty}\hat{f}(\xi)$$

We define a function on $\R$
$$s_b(c)=
\begin{cases}
(1+\frac{c}{b})^ke^{-ct}& \textbf{if } c\geq -b\\
0& \textbf{otherwise}
\end{cases}
$$ 
Then 
$$\hat{s_b}(\xi)=b^{-k}t^{-1-k}e^{tb+2i\pi b\xi}k!(1+\frac{2i\pi \xi}{t})^{-1-k}$$
And the Poisson summation formula applies to $s_b$\cite{pinsky} gives
$$f_b(t)=b^{-k}t^{-1-k}e^{tb}k!\sum_{\xi=-\infty}^{\infty}(1+\frac{2i\pi \xi}{t})^{-1-k}$$
Therefore
$$\rho_{0,k+1}=\frac{k}{2\pi}(1+\sum_{\xi=1}^{\infty}(1+(\frac{2\pi\xi }{t} )^2)^{-(k-1)/2}2\cos[ (k+1)\theta(t,\xi)])$$
where $\theta(t,\xi)$ satisfies $\cos\theta=\frac{1}{\sqrt{1+(\frac{2\pi\xi }{t} )^2}}$ and $\sin\theta_t=\frac{2\pi\xi/t}{\sqrt{1+(\frac{2\pi \xi}{t} )^2}}$.

 We then use the integral 
 $$\int_{1}^{\infty}(1+(\frac{2\pi \xi}{t} )^2)^{-(1+k)/2}\xi d\xi=\frac{t^2}{2\pi^2 (k-1)}(1+(\frac{2\pi }{t} )^2)^{-(k-1)/2} $$ to bound the summation, so we have

	

So we have proved theorem \ref{outside}.

\

We then turn to the proof of theorem \ref{middle}.

We will use the following basic lemma.
\begin{lem}\label{lemconcave}
	Let $f(x)$ be a concave function. Suppose $f'(x_0)<0$, then we have
	$$\int_{x_0}^\infty e^{f(x)}dx\leq\frac{e^{f(x_0)}}{-f'(x_0)}$$
\end{lem}

\begin{prop}
	For $b\leq \sqrt{k}\log k$,
	$$\gamma_b>f_b(\frac{k}{b})\geq (1-\frac{8(\log k)^4}{k})\gamma_b-(6+2e^{\frac{8(\log k)^3}{3\sqrt{k}}})k^{-2\log k}$$
	where $\gamma_b=\sum_{c=-\infty}^{\infty}e^{-kc^2/(2b^2)}$
\end{prop}
\begin{proof}
	Let $c_0=2\frac{b\log k}{\sqrt{k}}$. 
	We first truncate the summation $f_b(\frac{k}{b})=\sum_{c=-b}^{\infty}(1+\frac{c}{b})^ke^{-ck/b}$ to $\sum_{c=-c^0}^{c_0}(1+\frac{c}{b})^ke^{-ck/b}$. By doing so, we introduce errors:
	\begin{eqnarray*}
		\sum_{c=c^0}^{\infty}(1+\frac{c}{b})^ke^{-ck/b}&\leq& (1+\frac{c_0}{b})^ke^{-c_0k/b}/(1-e^{-\frac{kc_0}{(b+c_0)b}})\\
		&\leq&2(1+\frac{\log k}{\sqrt{k}})^ke^{-\sqrt{k}\log k}\\
		&\leq&2e^{-2(\log k)^2+\frac{8(\log k)^3}{3\sqrt{k}}}
	\end{eqnarray*}
	and 
	\begin{eqnarray*}
		\sum_{c=-b}^{-c_0}(1+\frac{c}{b})^ke^{-ck/b}&\leq& (1-\frac{c_0}{b})^ke^{c_0k/b}/(1-e^{-\frac{kc_0}{(b-c_0)b}})\\
		&\leq&2(1-\frac{\log k}{\sqrt{k}})^ke^{\sqrt{k}\log k}\\
		&\leq&2e^{-2(\log k)^2}
	\end{eqnarray*}
	Then we use the estimate for $x<1$,
	$$0>(1+x)e^{-x+x^2/2}+(1-x)e^{x+x^2/2}-2>-x^4$$
	, to estimate:
	\begin{eqnarray*}
		\sum_{c=-c_0}^{c_0}(1+\frac{c}{b})^ke^{-ck/b}&\geq& \sum_{c=-c^0}^{c_0}(1-\frac{kc^4}{2b^4})e^{-kc^2/(2b^2)}\\
		&\geq& \sum_{c=-c^0}^{c_0}(1-\frac{8(\log k)^4}{k})e^{-kc^2/(2b^2)}
	\end{eqnarray*}
and 
	\begin{eqnarray*}
	\sum_{c=-c_0}^{c_0}(1+\frac{c}{b})^ke^{-ck/b}&\leq& \sum_{c=-c^0}^{c_0}e^{-kc^2/(2b^2)}
\end{eqnarray*}
	Then the difference between $\sum_{c=-c^0}^{c_0}e^{-kc^2/(2b^2)}$ and $\sum_{c=-\infty}^{\infty}e^{-kc^2/(2b^2)}$ is less than $4e^{-2(\log k)^2}$. Then by putting them all together we get the conclusion.
\end{proof}
For $t\in (\frac{k}{b+1},\frac{k}{b})$, we need to look at $(g(d))^kf_b(t)$. Then we use the substitution $u=1-d$ to get
$$(de^{1-d})^k(1+\frac{c}{b})^ke^{-ct}=e^{k(\log (1+\frac{c}{b})-\frac{c}{b})+\frac{ck}{b}u+ku+k\log(1-u)}$$
We can take derivatives to see that the exponent is a concave function of $c$. We also notice that $u\in (0,\frac{1}{b+1})$. So if we truncate the summation again using $c_0=\frac2{b\log k}{\sqrt{k}}$, and assume $b\geq \frac{\sqrt{k}}{\log k}$, we should have similar results as in the proof of last proposition. So we follow the ideas there. 

\begin{prop}
	For $b\leq \sqrt{k}\log k$, and $t\in (\frac{k}{b+1},\frac{k}{b})$
	$$(g(d))^kf_b(t)\geq  (1-\frac{(\log k)^3}{3\sqrt{k}})\gamma_b(u)-12e^{-\frac{k}{2}(\frac{\log k}{\sqrt{k}}-u)^2}$$
	$$(g(d))^kf_b(t)\leq  (1+\frac{(\log k)^3}{3\sqrt{k}})\gamma_b(u)$$
	,	where $u=1-d$ and $\gamma_b(u)=\sum_{c=-\infty}^{\infty}e^{-\frac{k}{2}(\frac{c}{b}-u)^2}$
\end{prop}
\begin{proof}
	We denote by $g_b(c,u)=k(\log (1+\frac{c}{b})-\frac{c}{b})+\frac{ck}{b}u+ku+k\log(1-u)$. 
	Then $\frac{d}{dc}g_b(c,u)=-\frac{kc}{b(b+c)}+\frac{ku}{b}$
	\begin{eqnarray*}
		\sum_{c=c_0}^{\infty}e^{g_b(c,u)}&\leq& e^{g_b(c_0,u)}/(1-e^{-\frac{kc_0}{b(b+c_0)}+\frac{ku}{b}})\\
		&\leq& 2e^{g_b(c_0,u)}
	\end{eqnarray*}
	and 
	\begin{eqnarray*}
		\sum_{c=-b}^{-c_0}e^{g_b(c,u)}&\leq& e^{g_b(-c_0,u)}/(1-e^{-\frac{kc_0}{b(b-c_0)}+\frac{ku}{b}})\\
		&\leq& 2e^{g_b(-c_0,u)}
	\end{eqnarray*}
	We then use the inequality
	$e^{Log[1 - x] - x + x^2/2} - 1 + x^3/2\geq 0$ for $0.4>x\geq 0$ to get
	\begin{eqnarray*}
		e^{g_b(c,u)}&\geq& (1-\frac{k|c|^3}{3b^3})e^{-\frac{k}{2}(\frac{c}{b}-u)^2}
	\end{eqnarray*}
and since $u\geq 0$,
	\begin{eqnarray*}
	e^{g_b(c,u)}&\leq& (1+\frac{k|c|^3}{3b^3})e^{-\frac{k}{2}(\frac{c}{b}-u)^2}
\end{eqnarray*}
	So 
	\begin{eqnarray*}
		\sum_{c=-c_0}^{c_0}	e^{g_b(c,u)}&\geq& (1-\frac{k|c_0|^3}{3b^3})	\sum_{c=-c_0}^{c_0}e^{-\frac{k}{2}(\frac{c}{b}-u)^2}
	\end{eqnarray*}
and 
	\begin{eqnarray*}
	\sum_{c=-c_0}^{c_0}	e^{g_b(c,u)}&\leq& (1+\frac{k|c_0|^3}{3b^3})	\sum_{c=-c_0}^{c_0}e^{-\frac{k}{2}(\frac{c}{b}-u)^2}
\end{eqnarray*}
	While $$\sum_{c>c_0}e^{-\frac{k}{2}(\frac{c}{b}-u)^2}\leq e^{-\frac{k}{2}(\frac{\log k}{\sqrt{k}}-u)^2}/(1-e^{-k(\frac{\log k}{\sqrt{k}}-u)})\leq 2e^{-\frac{k}{2}(\frac{\log k}{\sqrt{k}}-u)^2}$$
	and 
	$$\sum_{c<-c_0}e^{-\frac{k}{2}(\frac{c}{b}-u)^2}\leq e^{-\frac{k}{2}(-\frac{\log k}{\sqrt{k}}-u)^2}/(1-e^{-k(\frac{\log k}{\sqrt{k}}+u)})\leq 2e^{-\frac{k}{2}(\frac{\log k}{\sqrt{k}}+u)^2}$$
	So we have
	\begin{eqnarray*}
		\sum_{c=-b}^{\infty}e^{g_b(c,u)}&\geq& (1-\frac{(\log k)^3}{3\sqrt{k}})\gamma_b(u)-2e^{-\frac{k}{2}(\frac{\log k}{\sqrt{k}}-u)^2}-2e^{-\frac{k}{2}(\frac{\log k}{\sqrt{k}}+u)^2}\\
		&&-2e^{g_b(c_0,u)}-2e^{g_b(-c_0,u)}\\
		&\geq & (1-\frac{(\log k)^3}{3\sqrt{k}})\gamma_b(u)-12e^{-\frac{k}{2}(\frac{\log k}{\sqrt{k}}-u)^2}
	\end{eqnarray*}
\end{proof}

So we have proved theorem \ref{middle}

\section{punctured Riemann surfaces }
We will use the same setting as that in \cite{AMM}, as described in the introduction.

To obtain global sections of $L^k$ from local ones, we need to use H\"ormander's $L^2$ estimate. The following lemma is well-known, see for example \cite{Tian1990On}. 

\begin{lem}\label{lemHorm}
	Suppose $(M,g)$ is a complete \kahler manifold of complex dimension $n$, $\mathcal L$ is a line bundle on $M$ with hermitian metric $h$. If 
	$$\langle-2\pi i \Theta_h+Ric(g),v\wedge \bar{v}\rangle_g\geq C|v|^2_g$$
	for any tangent vector $v$ of type $(1,0)$ at any point of $M$, where $C>0$ is a constant and $\Theta_h$ is the curvature form of $h$. Then for any smooth $\mathcal L$-valued $(0,1)$-form $\alpha$ on $M$ with $\bar{\partial}\alpha=0$ and $\int_M|\alpha|^2dV_g$ finite, there exists a smooth $\mathcal L$-valued function $\beta$ on $M$ such that $\bar{\partial}\beta=\alpha$ and $$\int_M |\beta|^2dV_g\leq \frac{1}{C}|\alpha|^2dV_g$$
	where $dV_g$ is the volume form of $g$ and the norms are induced by $h$ and $g$.
\end{lem}
In our setting, we can take the constant $C$ in the lemma above to be $k\epsilon+\lambda$. We have basic bounds for $\epsilon$  and $\lambda$, namely
$$\epsilon\leq d/v,\quad \lambda\leq (2-2g-N)/v$$
where $2\pi v$ is the volume of $X_D$ with volume form $\omega$, $g$ is the genus of $X$, and $N$ is the degree of $D$.

\

The following two assumptions are dependent(weaker) on the assumptions 1,2 and 3, but we list them below for convenience.
\begin{itemize}
	\item[4. ] $k^{1/4}\sqrt{k\epsilon+\lambda}(\frac{R}{2})^{\frac{k}{-2e\log(R/2)}}\leq 1$
	\item[5. ] $\frac{\log k}{4}\geq\log (-2e\log\frac{R}{2})-\frac{1}{k}\log(\epsilon+\frac{\lambda}{k})$.
\end{itemize}
\begin{rem}
	\begin{itemize}
		\item[$\bullet$]	Assumption 4 is weaker than the combination of assumption 1 and assumption 2. 
		\item[$\bullet$]	Assumption 5 is weaker than the combination of assumption 2 and assumption 3.
		\item[$\bullet$]	One reason for the number 23190 in assumption 3 is to make $k$ satisfy the inequality $k>(1+\frac{1}{4}\log k)^8$.
	\end{itemize}
	
\end{rem}

We fix a choice of smooth function $\chi(x)$ satisfying:
\begin{itemize}
	\item[$\bullet$] $\chi(x)=1$ for $x\leq 1/2$.
	\item[$\bullet$] $\chi(x)=0$ for $x\geq 1$.
	\item[$\bullet$] $\chi'(x)\leq 3$ for all $x$.
\end{itemize}
Then for each positive $R$, we define $\chi_R(x)=\chi(\frac{x}{R})$. In particular, we have $\chi_R'(x)\leq \frac{3}{R}$. Then with the trivialization of $L$ on $V_i$ as in assumption $\alpha$, we consider $\chi_{R_i}\tau_az_i^a$ as a global section of $L^k$ for each integer $a\geq 1$. Then we solve the $\dbar-$equation $$\dbar\beta=(\dbar \chi_{R_i})\tau_az_i^a$$
Then by lemma \ref{lemHorm}, we can find $\beta_{i,a}$ such that
$$\int_{X_D} |\beta_{i,a}|_h^2\omega\leq \frac{1}{k\epsilon+\lambda}\int_{X_D}|\alpha_{i,a}|^2_h\omega$$  
where $\alpha_{i,a}=(\dbar \chi_{R_i})\tau_az_i^a$.

We denote by $\bar{s_i}_a=\chi_{R_i}\tau_az^a-\beta_{i,a}$ for $a\leq  k^{3/4}$, $i\in \{1,\cdots,N\}$. Clearly $\bar{s}_{i,a}\in \hcal_k$. 

\

Then we will focus on the neighborhood $V_i$ of each singularity point $a_i\in D$, and for simplicity, we will fix an index $i$ and write $R=R_i$ and $\bar{s}_{i,a}=\bar{s}_a$ for $a\leq k^{3/4}$, etc.. 

\

Then we apply the Gram-Schmidt process to $\bar{s}_a$ as $a$ increases from $1$ to $k^{3/4}$, producing sections $s_{a}$. We then add more sections $s_\gamma$, $\gamma>k^{3/4}$, so that together they form an orthonormal basis of $\hcal_k$. And the Bergman kernel of $\hcal_k$ is
$$\rho_k=\sum_{a\geq 1} |s_a|_h^2$$
Recall that we have
\begin{equation}
	\rho_{0,k}=(\log \log \frac{1}{|z|^2})^k\sum_{a=1}^{\infty}|\tau_az^a|^2
\end{equation}   
, where $\tau_a^2=\frac{a^{k-1}}{2\pi (k-2)!}$.

We can then calculate:
\begin{eqnarray*}
	\int_{|z|\geq R/2}|\tau_az^a|^2_h\omega &=&2\pi\tau_a^2\int_0^{-2\log \frac{R}{2}}e^{(k-2)\log t-at}dt\\
	&\leq &2\pi\tau_a^2(\frac{R}{2})^{2a}(-2\log \frac{R}{2})^{k-1}
\end{eqnarray*}
Then using the inequality $$k!\geq \sqrt{2\pi }k^{k+1/2}e^{-k}$$
we get that 
$$\int_{|z|\geq R/2}|\tau_az^a|^2_h\omega \leq \frac{k-1}{5\sqrt{k}}(\frac{-2ea\log(R/2)}{k})^{k-1}(\frac{R}{2})^{2a}$$
We will denote by $A_a=\frac{k-1}{5\sqrt{k}}(\frac{-2ea\log(R/2)}{k})^{k-1}(\frac{R}{2})^{2a}$.

\begin{eqnarray*}
	\parallel \alpha_a\parallel^2 &=&	\int |(\dbar \chi_R)\tau_az^a|^2_h\omega\\
	 &\leq &\frac{18\pi\tau_a^2}{R^2}\int_{(R/2)^2}^{R^2}xt^2x^at^{k-2}x^{-1}dx\\
	 &=&\frac{18\pi\tau_a^2}{R^2}\int_{-2\log R}^{-2\log (R/2)}e^{-(a+1)t+k\log t}dt\\
	&\leq &\frac{18\pi\tau_a^2\log 4}{R^2} (\frac{R}{2})^{2a+2}(-2\log \frac{R}{2})^k\\
		&\leq &\frac{k^{1/2}(k-1)}{a}(\frac{-2ea\log(R/2)}{k})^{k}(\frac{R}{2})^{2a}
\end{eqnarray*}
We will denote by $B_a=\frac{k^{1/2}(k-1)}{a(k\epsilon+\lambda)}(\frac{-2ea\log(R/2)}{k})^{k}(\frac{R}{2})^{2a}$.

\

The ratio
$$\frac{B_{a+1}}{B_a}=(\frac{a+1}{a})^{k-1}(\frac{R}{2})^2$$
says that $B_a$ is an increasing function of $a$ as long as $a\leq \frac{k-1}{-2\log (R/2)}$. In particular, when $a=\frac{k}{-2e\log(R/2)}$, we have 
$$B_a\leq \frac{-2e\log(R/2)(k-1)}{k^{1/2}(k\epsilon+\lambda)}(\frac{R}{2})^{\frac{k}{-e\log(R/2)}}$$
which is clearly very small for $k$ large enough. Also, 
$$A_a\leq B_a^{1/2}$$
when $a\leq \frac{k}{-2e\log(R/2)}$ and $k$ satisfies assumptions 2 and 3 .

\

We denote by $E(a_0)=\sum_{a\leq a_0}B_a^{1/2}$. Then
for $a_0\leq  \frac{k-1}{-2\log (R/2)}$, we have
	\begin{eqnarray*}
		E(a_0)&=&\frac{k^{1/4}(k-1)^{1/2}}{\sqrt{(k\epsilon+\lambda)}}(\frac{-2e\log \frac{R}{2}}{k})^{k/2}\sum_{a\leq a_0}(\frac{R}{2})^aa^{(k-1)/2}\\
		&\leq &\frac{k^{1/4}(k-1)^{1/2}}{\sqrt{(k\epsilon+\lambda)}}(\frac{-2e\log \frac{R}{2}}{k})^{k/2}(\frac{R}{2})^{a_0}a_0^{(k+1)/2}\\
		&=&a_0B_{a_0}^{1/2}
	\end{eqnarray*}

\begin{lem}Assume assumption 1. For $a<k^{3/4}$, 
\begin{equation}
	|\parallel \bar{s}_a\parallel^2-1|\leq 4B_a^{1/2}
\end{equation}
\end{lem}
\begin{proof}
	\begin{eqnarray*}
		\parallel \bar{s}_a\parallel^2=\parallel \chi_R\tau_az^a\parallel^2+\parallel \beta_a\parallel^2+2\Re \int \chi_R\tau_az^a\bar{\beta_a}\omega
	\end{eqnarray*}
where $1>\parallel \chi_R\tau_az^a\parallel^2>1-A_a$, and $\parallel \beta_a\parallel^2<B_a$. So 
$|\Re \int \chi_R\tau_az^a\bar{\beta_a}\omega|\leq B_a^{1/2}$. And the conclusion follows with the assumption on $k$.
\end{proof}
Similarly, one can easily prove the following: 
\begin{lem}
	For $a<b\leq k^{3/4}$, 
	$$|<\bar{s}_a,\bar{s}_b>|\leq 3B_b^{1/2}$$
\end{lem}

We denote by
$$\mu_b^{-2}=\parallel \bar{s}_b\parallel^2-\sum_{a<b}|<\bar{s}_b,s_a>|^2 $$
Then $$\frac{s_b}{\mu_b}=\bar{s}_b-\sum_{a<b}<\bar{s}_b,s_a>s_a$$

\begin{lem}Assume that $B_{k^{3/4}}\leq k^{-3/2}/256$ then
	for $a<b\leq k^{3/4}$, 
	\begin{eqnarray*}
		|<\bar{s}_b,s_a>|&\leq& 4\sqrt{B_b}\\
		|\mu_b^{-2}-1|&\leq& 5B_b^{1/2}
	\end{eqnarray*}
\end{lem}
\begin{rem}
	This assumption on $k$ is weaker than assumptions 1 to 3 combined.
\end{rem}
\begin{proof}
	We use induction on $b$. When $b=1$, it has been proved. 
	Now $$<\bar{s}_b,s_a>=\mu_a(<\bar{s}_a,\bar{s}_b>+\sum_{i=1}^{a-1}<\bar{s}_a,s_i><\bar{s}_b,s_i>)$$
	By induction, $|\mu_a-1|\leq 3B_a^{1/2}$, so
	\begin{eqnarray*}
		|<\bar{s}_b,s_a>|&\leq& (1+5B_a^{1/2})[3B_b^{1/2}+\sum_{i=1}^{a-1}16B_a^{1/2}B_b^{1/2}]\\
	&\leq& (1+5B_a^{1/2})[3B_b^{1/2}+16(a-1)B_{a}^{1/2}B_b^{1/2}]	\\
	&\leq & 4B_b^{1/2}
	\end{eqnarray*}
Then
\begin{eqnarray*}
	|\mu_b^{-2}-1|&\leq& 4B_b^{1/2}+\sum_{i=1}^{a-1}16B_b\\
	&\leq & 5B_b^{1/2}
\end{eqnarray*}
	
\end{proof}

For each section $\alpha$ of $L^k$ that is well defined for $|z|\leq R/2$, we define 
$$\parallel \alpha \parallel_R=(\int_{|z|\leq R/2}|\alpha|_h^2\omega)^{1/2} $$

\begin{nota}
	For simplicity of notations, in the following we will use absolute value to denote the point-wise norm of sections of $L^k$, namely, when $\alpha$ is a section of $L^k$, $|\alpha|$ means $|\alpha|_h$.
\end{nota}

\begin{lem}\label{linfinity}
	Assume assumptions 1 to 3.
	Let $\alpha$ be a holomorphic function on $U$, considered as a local section of $L^k$ .
	Assume $\parallel \alpha \parallel_R^2\leq \nu$. Then for all $|z|^2\leq e^{-(k-2)^{3/8}}$, we have 
	$$|\alpha|^2\leq 4\nu \rho_{0,k}$$
\end{lem}
\begin{proof}
Write $\alpha=\sum_{i=1}^{\infty}c_i\tau_iz^i$. Then the assumption on the $L_2$ norms gives $$\sum_{i=1}^{\infty}|c_i|^2(1-A_i)\leq \nu$$
With the hypothesis of the lemma, we have that $A_i\leq \frac{2}{3}$
when $i\leq a_0= \frac{k-2}{-2\log(R/2)}$. So 
$$\sum_{1\leq i\leq a_0}|c_i|^2\leq 3\nu$$
Therefore by Cauchy-Schwartz inequality $$|\sum_{1\leq i\leq a_0}c_i\tau_iz^i|^2\leq 3\nu \rho_{0,k}$$
For $a\geq a_0$, we use the technique used in \cite{SunSun}. Let $b=(k-2)^{5/8}$,  $q_a(z)=|\frac{c_a\tau_az^a}{\tau_bz^b}|$ and $\gamma$ be the value of $q_a(z)$ evaluated at the points where $\log \frac{1}{|z|^2}=(k-2)^{3/8}$.
We define
$$T=\{z: (k-2)^{3/8}/2\leq \frac{1}{|z|^2}\leq (k-2)^{3/8} \}$$
Then we have
$$\int_T |\tau_bz^b|^2_h\omega \geq 1/3$$
, therefore
$$\int_T |c_a\tau_az^a|^2_h\omega \geq \gamma_a^2/3$$
On the other hand, we can estimate
\begin{eqnarray*}
	\int_T |c_a\tau_az^a|^2_h\omega &=&2\pi|c_a\tau_a|^2\int_{(k-2)^{3/8}/2}^{(k-2)^{3/8}}e^{(k-2)\log t-at}dt\\
	&\leq& 2\pi|c_a\tau_a|^2\frac{(k-2)^{3/8}}{2^{k-1}}(k-2)^{3(k-2)/8}e^{-a(k-2)^{3/8}/2}\\
	&=&\pi|c_a\tau_a|^22^{2-k}(k-2)^{3(k-1)/8}e^{-a(k-2)^{3/8}/2}
\end{eqnarray*}

Also we have
\begin{eqnarray*}
	\nu &\geq & \int_{|z|\leq R/2}|c_a\tau_az^a|^2_h\omega\\
	&\geq & 2\pi|c_a\tau_a|^2(-2\log (R/2))e^{(k-2)\log (-4\log (R/2))+4a\log(R/2)}\\
	&=& \pi|c_a\tau_a|^2(\frac{R}{2})^{4a}(-4\log\frac{R}{2})^{k-1}
\end{eqnarray*}
All combined, we get that
$$\frac{\gamma_a^2}{3}\leq \nu (\frac{R}{2})^{-4a}a^{4-3k}[\frac{(k-2)^{3/8}}{-\log(R/2)}]^{k-1}e^{-a(k-2)^{3/8}/2}$$

So when $|z|^2\leq e^{-(k-2)^{3/8}}$, for $a\geq a_0$, we have
$$|c_a\tau_az^a|^2\leq 3\nu \rho_{0,k}(\frac{R}{2})^{-4a}2^{4-3k}[\frac{(k-2)^{3/8}}{-\log(R/2)}]^{k-1}e^{-a(k-2)^{3/8}/2}$$
So
\begin{eqnarray*}
	|\sum_{a\geq a_0}c_a\tau_az^a|^2 &\leq& (\sum_{a\geq a_0}|c_a\tau_az^a|)^2  \\
	&\leq&3\nu|\tau_bz^b|^2[\frac{(k-2)^{3/8}}{-\log(R/2)}]^{k-1}2^{4-3k}	(\sum_{a\geq a_0}(\frac{R}{2})^{-2a}e^{-a(k-2)^{3/8}/4})^2\\
	&\leq &3\nu |\tau_bz^b|^2[\frac{(k-2)^{3/8}}{-\log(R/2)}]^{k-1}2^{4-3k}u^{2a_0}/(1-u)^2\\
		&\leq &\frac{1}{24}\nu |\tau_bz^b|^2
\end{eqnarray*}
where $u=(\frac{R}{4})^{-2}e^{-(k-2)^{3/8}/4}$, and the last inequality holds with the given assumption that 
$$(k-2)^{3/8}\geq -8\log(R/4)+4\log 2$$
which is weaker than assumptions 2 and 3 combined.

All together we get the conclusion.
\end{proof}

\begin{lem}\label{l2}
	Assume $B_{k^{3/4}}\leq k^{-3/2}/256$.
	For $a\leq k^{3/4}$, 
	$$\parallel s_a-\tau_az^a \parallel_R^2\leq  125aB_a$$

\end{lem}

\begin{proof}
	When $|z|\leq R/2$, $\bar{s}_a=\tau_az^a-\beta_a$. So $s_a=\mu_a(\tau_az^a-\beta_a)-\mu_a\sum_{i=1}^{a-1}<\bar{s}_a,s_i>s_i$. 
	We denote by $\eta_a=\beta_a+\sum_{i=1}^{a-1}<\bar{s}_a,s_i>s_i$
	Then $$	|s_a-\tau_az^a|^2\leq |\mu_a-1|^2|\tau_az^a|^2+|\mu_a|^2|\eta_a|^2+2|(\mu_a-1)\mu_a\eta_a\tau_az^a|$$
	and 
	\begin{eqnarray*}
			\parallel \eta_a \parallel_R^2&\leq &	\parallel \beta_a \parallel_R^2+\mu_a^2	\parallel \sum_{i=1}^{a-1}<\bar{s}_a,s_i>s_i\parallel_R^2\\&&+2\mu_a	\parallel \beta_a \parallel_R	\parallel \sum_{i=1}^{a-1}<\bar{s}_a,s_i>s_i \parallel_R\\
			&\leq &	B_a+\mu_a^2	 \sum_{i=1}^{a-1}|<\bar{s}_a,s_i>|^2+2\mu_a B_a^{1/2}(\sum_{i=1}^{a-1}|<\bar{s}_a,s_i>|^2)^{1/2}\\
				&\leq &	B_a+(1+10B_a^{1/2})	 16(a-1)B_a+2\mu_a B_a^{1/2}(16(a-1)B_a)^{1/2}\\
					&\leq &	25aB_a
	\end{eqnarray*}
	So
	\begin{eqnarray*}
	\parallel s_a-\tau_az^a \parallel_R^2&\leq&  |\mu_a-1|^2+|\mu_a|^2|	\parallel \eta_a \parallel_R^2+2(\mu_a-1)\mu_a	\parallel \eta_a \parallel_R\\
	&\leq& 25B_a+(1+10B_a^{1/2})25aB_a+10B_a^{1/2}(1+5B_a^{1/2})5\sqrt{a}B_a^{1/2}\\
	&\leq& 125aB_a
	\end{eqnarray*}
\end{proof}
\begin{lem}
		Assume assumptions 1 to 3, then for $a\leq k^{3/4}$, 
	$$	|\sum_{a\leq k^{3/4}}|s_a|^2-\sum_{a\leq k^{3/4}}|\tau_az^a|^2|	\leq 50k^{3/4}B_{k^{3/4}}^{1/2}\rho_{0,k}$$
	Here, we have considered $s_a$ as holomorphic functions.
\end{lem}

\begin{proof}

Combining lemmas \ref{linfinity} and  \ref{l2}, we see that for $a\leq k^{3/4}$,
$$||s_a|^2-|\tau_az^a|^2|\leq 500aB_a\rho_{0,k}+4\sqrt{125aB_a\rho_{0,k}}|\tau_az^a|$$ 
Therefore if we denote by $D(k^{3/4})=	|\sum_{a\leq k^{3/4}}|s_a|^2-\sum_{a\leq k^{3/4}}|\tau_az^a|^2|$, then 
\begin{eqnarray*}
	D(k^{3/4})&\leq & 500\rho_{0,k}\sum_{a\leq k^{3/4}}aB_a+4\sqrt{125\rho_{0,k}}\sum_{a\leq k^{3/4}}|\tau_az^a|(aB_a)^{1/2}\\ 
	&\leq &500\rho_{0,k}\sum_{a\leq k^{3/4}}aB_a+4\sqrt{125\rho_{0,k}}(\sum_{a\leq k^{3/4}}|\tau_az^a|^2\sum_{a\leq k^{3/4}}aB_a)^{1/2}\\
	&\leq &500\rho_{0,k}k^{3/2}B_{k^{3/4}}+4\sqrt{125}\rho_{0,k}k^{3/4}B_{k^{3/4}}^{1/2}\\
	&\leq &50k^{3/4}B_{k^{3/4}}^{1/2}\rho_{0,k}
\end{eqnarray*}	
, where in the last inequality we used the assumption that $B_{k^{3/4}}\leq \frac{k^{-3/2}}{100^2}$, which is weaker than the assumptions 1 to 3 combined.
\end{proof}

Next, we consider the sections $s_j$ for $j>a_2=k^{3/4}$. 
\begin{lem}\label{lem-others}
	Assume assumptions 1 to 3. We have
	$$	\sum_{i>k^{3/4}}|\tau_iz^i|^2\leq 2k^{(k-1)/8}e^{k-k^{9/8}}|\tau_bz^b|^2$$
And for $j>{a_2}$, we have $$|s_j|^2\leq 18{a_2}B_{a_2}\rho_{0,k}$$
	
\end{lem}
\begin{proof}
	The condition that $s_j$ is orthogonal to all sections $s_a$, $a\leq {a_2}$, is equivalent to the condition that $s_j$ is orthogonal to all sections $\bar{s}_a$, $a\leq {a_2}$. So we have that $$<s_j,\chi_R\tau_az^a>=<s_j,\beta_a>$$
	 If we expand $s_j=\sum_{i=1}^{\infty}c_i\tau_iz^i$, then the left hand side is just $$c_a|\tau_a|^2\int_{|z|\leq R}|z^a|^2_h\omega  $$
	  Therefore we have
	$$|c_a|\parallel \tau_az^a \parallel_R^2\leq \parallel \beta \parallel_R$$
	So $|c_a|^2\leq 2B_a$ for $a\leq {a_2}$. We also have that
	$$\sum_{i=1}^{\infty}|c_i|^2\int_{|z|\leq R}|\tau_iz^i|_h^2\omega\leq 1$$
	Since we know that $\int_{|z|\leq R}|\tau_iz^i|_h^2\omega\geq 1/3$ for $i\leq a_1=\frac{k-2}{-2\log R}$, we have $$\sum_{i=1}^{a_1}|c_i|^2\leq 3$$
	So we have 
	\begin{equation}
	|\sum_{i=1}^{{a_2}}c_i\tau_iz^i|^2\leq \sum_{i=1}^{{a_2}}2B_i\rho_{0,k}\leq 2{a_2}B_{a_2}\rho_{0,k}
	\end{equation}
	When $i>k^{3/4}$ and $t=\log\frac{1}{|z|^2}\geq k^{3/8}$, we have
	$$\sum_{i>k^{3/4}}	\frac{|\tau_iz^i|^2}{|\tau_bz^b|^2}=\frac{e^{-bt}}{b^{k-1}}\sum_{i>k^{3/4}}	i^{k-1}e^{-it}$$
	To estimate the summation $\sum_{i>k^{3/4}}	i^{k-1}e^{-it}$, we write
	$i^{k-1}e^{-it}=e^{(k-1)\log i-it}$. Take the derivative of the exponent with respect to $i$, we get
	$\frac{k-1}{i}-it<0$, meaning that $i^{k-1}e^{-it}$ is decreasing as $i$ increases, and it decreases faster than the power series $\sum e^{i((k-1)/k^{3/4}-k^{3/8})}$. So we have  
	\begin{eqnarray*}
		\sum_{i>k^{3/4}}	\frac{|\tau_iz^i|^2}{|\tau_bz^b|^2}&\leq &(\frac{k^{3/4}}{k^{5/8}})^{k-1}e^{-k^{3/8}(k^{3/4}-k^{5/8})}/(1-e^{k^{1/4}-k^{3/8}})\\
		&\leq &2k^{(k-1)/8}e^{k-k^{9/8}}
	\end{eqnarray*}
	Therefore, we get
	\begin{equation}\label{terms-inside}
		\sum_{i>k^{3/4}}|\tau_iz^i|^2\leq 2k^{(k-1)/8}e^{k-k^{9/8}}|\tau_bz^b|^2
	\end{equation}
		\begin{equation}
		|\sum_{i=k^{3/4}+1}^{a_1}c_i\tau_iz^i|^2\leq 6k^{(k-1)/8}e^{k-k^{9/8}}|\tau_bz^b|^2
	\end{equation}
	, for $z\in W_k$. 
	
	\
	
	Basically by repeating the argument above, one can also prove the following
		\begin{equation}\label{terms-inside-1}
\sum_{i>k^{3/4}}i|\tau_iz^i|^2\leq 2k^{(k+5)/8}e^{k-k^{9/8}}|\tau_bz^b|^2\\
\end{equation}
\begin{equation}\label{terms-inside-2}
|\sum_{i=k^{3/4}+1}^{a_1}ic_i\tau_iz^i|^2\leq 6k^{(k+11)/8}e^{k-k^{9/8}}|\tau_bz^b|^2
\end{equation}
		, for $z\in W_k$. 
	
	\
	
	When $i>a_1$, we use the property that $|\tau_iz^i|^2_h$ is decreasing as $t$ increases from $-2\log R$ to $\infty$. So 
	\begin{eqnarray*}
		\int_{|z|\leq R}|z^i|^2_h\omega &\geq & 2\pi\int _{-2\log R}^{-2\log R+1/2}e^{(k-2)\log t-it}dt\\
		&\geq &\pi e^{(k-2)\log t_1-it_1}
	\end{eqnarray*}
	, where $t_1=-2\log R+1/2$. Which implies 
	$$|c_i\tau_i|^2\pi e^{(k-2)\log t_1-it_1}\leq 1$$
	
	On the other hand, we have
	$$\int_{|z|\leq R/2}|z^i|^2_h\omega\leq 2\pi (-2\log \frac{R}{2})^{k-2}(\frac{R}{2})^{2i}/(i-\frac{k-2}{-2\log (R/2)})$$
	So
	$$	\parallel c_i\tau_iz^i \parallel_R^2 \leq \frac{2}{i-\frac{k-2}{-2\log (R/2)}} (\frac{-2\log \frac{R}{2}}{t_1})^{k-2}e^{i(1/2-\log 4)}$$
	We can then use the technique used in the proof of lemma \ref{linfinity} to get that
	\begin{eqnarray*}
		| c_i\tau_iz^i |^2 &\leq & \frac{6|\tau_bz^b|^2}{i-\frac{k-2}{-2\log (R/2)}} (\frac{-2\log \frac{R}{2}}{t_1})^{k-2}e^{-0.88i}\\
		&&\cdot (\frac{R}{2})^{-4i}2^{4-3k}[\frac{(k-2)^{3/8}}{-\log(R/2)}]^{k-1}e^{-i(k-2)^{3/8}/2}
	\end{eqnarray*}
	, for $z\in W_k$. Then we have
	\begin{eqnarray*}
		|\sum_{i>a_1} c_i\tau_iz^i |^2 &\leq &  (\sum_{i>a_1} |c_i\tau_iz^i |)^2  \\
			&\leq &|\tau_bz^b|^2	u_1(\sum_{i>a_1}(\frac{R}{2})^{-2i}e^{-i(0.88+(k-2)^{3/8}/4)})^2\\
			&= &|\tau_bz^b|^2	u_1u_2^{2(a_1)}/(1-u_2^2(1))^2
	\end{eqnarray*}
, where
$$u_1=\frac{3/2}{a_1-\frac{(k-2)}{-2\log(R/2)}}(\frac{2\log(R/2)}{2t_1})^{k-2}[\frac{(k-2)^{3/8}}{-\log(R/2)}]^{k-1}$$
, and 
$$u_2(i)=e^{-i(0.88+(k-2)^{3/8}/4+2\log(R/2))}$$
Now we need to compare the three bounds:
\begin{eqnarray*}
	I&=&2a_2B_{a_2}\\
	&=&2k^{(2-k)/4}\frac{k-1}{k}\frac{1}{\epsilon+\lambda/k}(-2e\log (R/2))^k(\frac{R}{2})^{2k^{3/4}}
\end{eqnarray*}

$II=6k^{(k-1)/8}e^{k-k^{9/8}}$
and $III=2u_1(u_2(a_1))^2$.

We can then compute:
$$\log(III)=-(k-2)b_1+b_2$$
where $b_1=\frac{(k-2)^{3/8}}{-4\log R}-\frac{3}{8}\log(k-2)+f_1(R)$, $b_2=\frac{-5}{8}\log(k-2)+f_2(R)$
with
\begin{eqnarray*}
	f_1(R)&=&\frac{2\log 2-0.22}{\log R}-(2+\log 2)-\log(1-4\log R)\\
	&\geq &\frac{1.17}{\log R}-2.7-\log(1-4\log R)
\end{eqnarray*} 
, and 
\begin{eqnarray*}
	f_2(R)&=&-\log(-\log(R/2))+\log\frac{(\log R)(\log (R/2))}{\log 2}+\log 6\\
	&=&\log(-\log R)-\log \log 2+\log 6 
\end{eqnarray*}
Also
$$\log(II)=k-k^{9/8}+\frac{k-1}{8}\log k$$
 Then one can see that in order to have $\log (II)-\log (III)>0$,
 we can require
 \begin{eqnarray*}
 	\frac{(k-2)^{3/8}}{-8\log R}&\geq &\frac{\log(k-2)}{4}\\
 	\frac{(k-2)^{3/8}}{-9\log R}&\geq& (k-2)^{1/8}
 \end{eqnarray*} 
, both of which are satisfied when $k-2\geq (-9\log R)^4$ and $k\geq 23190$.
Therefore $II>III$ under assumptions 2 and 3.

$\log (I)-\log(II)$ is $v_1-v_2$, where $$v_1=k\log(-2\log(R/2))+2k^{3/4}\log(R/2)+k^{9/8}-\frac{3k+1}{8}\log k$$
, and
$$v_2=\log\frac{k}{k-1}+\log(\epsilon+\frac{\lambda}{k})+\log 3$$. Then under the assumptions 1 to 3, one can prove that $\log (I)-\log(II)>0$, hence
$I>II$. 

So we have proved the lemma.
\end{proof}

\begin{thm}
	Assume assumptions 1 to 3, then for $z\in W_k$, 
	$$	|\sum_{a\geq 1}|s_a|^2-\rho_{0,k}|	\leq 50(1+d)k^{3/4}B_{k^{3/4}}^{1/2}\rho_{0,k}$$
\end{thm}
\begin{proof}
	Combining the lemmas above, we get
	\begin{eqnarray*}
		|\sum_{a\geq 1}|s_a|^2-\rho_{0,k}|&\leq & |\sum_{a\leq {a_2}}^{ }|s_a|^2-\sum_{a\leq {a_2}}|\tau_az^a|^2|+|\sum_{a> {a_2}}|s_a|^2|+|\sum_{a> {a_2}}|\tau_az^a|^2|         \\
		&\leq &\rho_{0,k}(50{a_2}B_{a_2}^{1/2}+18(dk-k^{3/4}){a_2}B_{a_2}+\frac{1}{3}{a_2}B_{a_2})\\
		&\leq & \rho_{0,k}(50{a_2}B_{a_2}^{1/2}(\frac{9}{25}dkB_{a_2}^{1/2}+1))\\
		&\leq & 50(1+9d/25 ){a_2}B_{a_2}^{1/2}\rho_{0,k}	
	\end{eqnarray*}
, where for the last inequality, we have used assumption 5, which implies $B_{a_2}\leq \frac{1}{k^2}$.
\end{proof}

By plugging in the formula for $B_{a_2}$
, we have proved theorem \ref{thm-riemannsurface}. 

\

Let $s=fe_L^k$, then $|s|^2_h=|f|^2t^k$, so
\begin{eqnarray*}
	\partial |s|^2_h&=&(f'\bar{f}t^k+|f|^2kt^{k-1}\frac{-1}{z})dz\\
	&=&\bar{f}t^{k-1}(tf'z-kf)\frac{dz}{z}
\end{eqnarray*}
So 
$$|\nabla |s|^2_h|=\sqrt{2}t^k|f||tf'z-kf|$$
Expand $f$ as $f=\sum_{i=1}^{\infty}c_i\tau_iz^i$, then $$tf'z-kf=\sum_{i=1}^{\infty}(it-k)c_i\tau_iz^i$$

\begin{lem}\label{nabla-others}
	For $j>{a_2}$,  under the fixed frame, we have
	$$|\nabla |s_j|^2_h|\leq 18\sqrt{2}ta_2^2B_{a_2}\rho_{0,k}$$
, for $z\in W_k$, where $|\cdot|$ is understood as module of complex numbers.
\end{lem}
\begin{proof}
		The arguments basically repeat with a little care the arguments in the proof of lemma \ref{lem-others}. We follow the notations there. Recall that $b=k^{5/8}$, $a_1=\frac{k-2}{-2\log R}$, $a_2=k^{3/4}$
For $z\in W_k$, we will get the following:
		\begin{eqnarray*}
			|\sum_{i=1}^{a_2}(it-k)c_i\tau_iz^i|^2&\leq& (\sum_{i=1}^{a_2}|c_i|^2(it-k)^2)(\sum_{i=1}^{a_2}|\tau_iz^i|^2)\\
			&\leq&(\sum_{i=1}^{a_2}2i^2B_i)(\sum_{i=1}^{a_2}|\tau_iz^i|^2)\\
			&\leq&2a_2^3t^2B_{a_2}(\sum_{i=1}^{a_2}|\tau_iz^i|^2)
		\end{eqnarray*}
	, and
		\begin{eqnarray*}
		|\sum_{i=a_2+1}^{a_1}(it-k)c_i\tau_iz^i|^2&\leq& (\sum_{i=a_2+1}^{a_1}|c_i|^2)(\sum_{i=a_2+1}^{a_1}|(it-k)\tau_iz^i|^2)\\
		&\leq&3t^2(\frac{e^{bt}}{b^{k-1}}\sum_{i=a_2+1}^{a_1}i^{k+1}e^{-it})|\tau_bz^b|^2\\
		&\leq&6t^2k^{(k+11)/8}e^{k-k^{9/8}}|\tau_bz^b|^2
	\end{eqnarray*}
	, and
			\begin{eqnarray*}
		|\sum_{i>a_1}(it-k)c_i\tau_iz^i|^2&\leq& t^2(\sum_{i>a_1}u_1(iu_2(i)))|\tau_bz^b|^2\\
		&\leq&t^2(u_1(a_1u_2(a_1)^2|\tau_bz^b|^2)/(1-e^{\log u_2(a_1)+1/a_1})^2\\
		&\leq&t^2u_1(a_1u_2(a_1))^2|\tau_bz^b|^2
	\end{eqnarray*}
Then one can show that $I$ is larger than both $II$ and $III$.
Therefore, we have proved that $$|tf'z-kf|^2\leq  18t^2a_2^3B_{a_2}\sum_{i=1}^{\infty}|\tau_iz^i|^2$$
Then we apply lemma \ref{lem-others}, to get the conclusion.
\end{proof}
 
\begin{rem}
	One can get better estimate than the preceding lemma by using lemmas such as \ref{lattice-inside}, if it is desired.
\end{rem}

The following lemma follows directly formula \ref{terms-inside-1}.
\begin{lem}
		$$|\nabla \sum_{a>a_2}|\tau_az^a|^2_h|\leq 2\sqrt{2}tk^{(k+5)/8}e^{k-k^{9/8}}\rho_{0,k}$$
\end{lem}

Also, when $a\leq a_2$, we write $s_a=g_a+\tau_az^a$. Then lemma \ref{l2} says that $\parallel g_a\parallel_R^2\leq 125aB_a$, and lemma \ref{linfinity} implies that 
$$|g_a|^2\leq 500aB_a R_k$$
, where $R_k\sum_{i=1}^{\infty}|\tau_iz^i|^2$
Then 
\begin{eqnarray*}
	\partial(|s_a|^2_h-|\tau_az^a|^2_h)&=&\frac{t^{k-1}dz}{z}[t(\bar{g}_ag_a'z+\tau_a\bar{z}^ag_a'z+a\tau_az^a\bar{g}_a)\\
	&&-k(\bar{g}_ag_a+\tau_a\bar{z}^ag_a+\tau_az^a\bar{g}_a)]
\end{eqnarray*}
We have seen that
$$|\bar{g}_ag_a+\tau_a\bar{z}^ag_a+\tau_az^a\bar{g}_a|\leq 500aB_aR_k+4\sqrt{125aB_aR_k}|\tau_az^a|$$
We also can prove, by repeating the proof of lemma \ref{linfinity}, that
$$|g_a'z|\leq k\sqrt{500aB_aR_k}$$
Therefore, 
\begin{eqnarray*}
	&&|\bar{g}_ag_a'z+\tau_a\bar{z}^ag_a'z+a\tau_az^a\bar{g}_a|\\
	&\leq &|\bar{g}_ag_a'z|+|\tau_a\bar{z}^ag_a'z|+|a\tau_az^a\bar{g}_a|\\
		&\leq &500kaB_aR_k+k\sqrt{500aB_aR_k}|\tau_az^a|+a\sqrt{500aB_aR_k}|\tau_az^a|\\
			&&\leq 24k\sqrt{aB_a}R_k
\end{eqnarray*}
So we have
\begin{eqnarray*}
	|\nabla(|s_a|^2_h-|\tau_az^a|^2_h)|&\leq&\sqrt{2}(24t+50)k\sqrt{aB_a}R_k
\end{eqnarray*}
Summing them up, we get 
\begin{lem}
		\begin{eqnarray*}
			|\nabla \sum_{a\leq a_2}(|s_a|^2_h-|\tau_az^a|^2_h)|&\leq&25\sqrt{2}tk\sqrt{a_2B_{a_2}}\rho_{0,k}
		\end{eqnarray*}
\end{lem}
So we have the theorem
\begin{thm}With the assumptions 1 to 3 and  $2k^{3/4}\geq \frac{\log(\epsilon+\lambda/k)}{\log (R/2)}$, we have
	$$	|\nabla (\rho_k-\rho_{0,k})|\leq (1+\frac{d}{k^2})25\sqrt{2}kt\sqrt{a_2B_{a_2}}\rho_{0,k} $$
\end{thm}
\begin{proof}
	Again, we need to compare the three terms 
	\begin{eqnarray*}
		I&=&25\sqrt{a_2B_{a_2}}\\
			II&=&2k^{(k-3)/8}e^{k-k^{9/8}}  \\
				III&=&18da_2^2B_{a_2}
	\end{eqnarray*}
With the additional assumption that $2k^{3/4}\geq \frac{\log(\epsilon+\lambda/k)}{\log (R/2)}$, one can show that $III\leq \frac{d}{k^3}I$ and $II\leq \frac{1}{k^2}I$, hence the conclusion. 

\end{proof}
By plugging in the formulas for $a_2$ and $B_{a_2}$, we get theorem \ref{nabla}

\bibliographystyle{plain}

\bibliography{references}

\end{document}